\documentclass[reqno,12pt]{amsart}

\usepackage{amssymb}

\newcommand{\N}{{\mathbb N}}

\theoremstyle{plain}
\newtheorem{lemma}{Lemma}
\newtheorem{theorem}{Theorem}

\theoremstyle{remark}
\newtheorem{problem}{Problem}

\newcommand{\refl}[1]{~\ref{l:#1}}
\newcommand{\reft}[1]{~\ref{t:#1}}
\newcommand{\refe}[1]{\eqref{e:#1}}
\newcommand{\refb}[1]{\cite{b:#1}}

\newcommand{\seq}{\subseteq}

\newcommand{\longc}{,\dotsc,}

\title[Integer sets with identical representation functions]%
  {Integer sets \\ with identical representation functions}

\author{Yong-Gao Chen}
\address{School of Mathematical Sciences and Institute of Mathematics,
         Nanjing Normal University, Nanjing 210023, China}
\email{ygchen@njnu.edu.cn}

\author{Vsevolod F. Lev}
\address{Department of Mathematics, The University of Haifa at Oranim,
  Tivon 36006, Israel}
\email{seva@math.haifa.ac.il}

\subjclass[2010]{11B34}
\keywords{Representation function}

\begin{document}
\baselineskip = 16pt

\begin{abstract}
We present a versatile construction allowing one to obtain pairs of integer
sets with infinite symmetric difference, infinite intersection, and identical
representation functions.
\end{abstract}

\maketitle

Let $\N_0$ denote the set of all non-negative integers. To every subset
$A\seq\N_0$ corresponds its representation function $R_A$ defined by
  $$ R_A(n) := | \{ (a',a'')\in A\times A\colon n=a'+a'',\ a'<a'' \} |; $$
that is, $R_A(n)$ is the number of unordered representations of the integer
$n$ as a sum of two distinct elements of $A$.

Answering a question of S\'ark\"ozy, Dombi \refb{d} constructed sets
$A,B\seq\N_0$ with infinite symmetric difference such that $R_A=R_B$. The
result of Dombi was further extended and developed in \refb{cw} (where a
different representation function was considered) and \refb{l} (a simple
common proof of the results from \refb{d} and \refb{cw} using generating
functions); other related results can be found in \cite{b:c,b:ct,b:q,b:t}.

The two sets constructed by Dombi actually \emph{partition} the ground set
$\N_0$, which makes one wonder whether one can find $A,B\seq\N_0$ with
$R_A=R_B$ so that not only the symmetric difference of $A$ and $B$, but also
their intersection is infinite. Tang and Yu \refb{ty} proved that if
 $A\cup B=\N_0$ and $R_A(n)=R_B(n)$ for all sufficiently large integer $n$,
then at least one cannot have $A\cap B=4\N_0$ (here and below $k\N_0$ denotes
the dilate of the set $\N_0$ by the factor $k$). They further conjectured
that, indeed, under the same assumptions, the intersection $A\cap B$ cannot
be an infinite arithmetic progression, unless $A=B=\N_0$. The main goal of
this note is to resolve the conjecture of Tang and Yu in the negative by
constructing an infinite family of pairs of sets $A,B\seq\N_0$ with $R_A=R_B$
such that $A\cup B=\N_0$, while $A\cap B$ is an infinite arithmetic
progression properly contained in $\N_0$. Our method also allows one to
easily construct sets $A,B\seq\N_0$ with $R_A=R_B$ such that both their
symmetric difference and intersection are infinite, while their union is
arbitrarily sparse and the intersection is \emph{not} an arithmetic
progression.

For sets $A,B\seq\N_0$ and integer $m$ let
 $A-B:=\{a-b\colon (a,b)\in A\times B\}$ and $m+A:=\{m+a\colon a\in A\}$.

The following basic lemma is in the heart of our construction.
\begin{lemma}\label{l:main}
Suppose that $A_0,B_0\seq\N_0$ satisfy $R_{A_0}=R_{B_0}$, and that $m$ is a
non-negative integer with $m\notin(A_0-B_0)\cup(B_0-A_0)$. Then, letting
  $$ A_1 := A_0\cup(m+B_0)\ \text{and}\ B_1:=B_0\cup(m+A_0), $$
we have $R_{A_1}=R_{B_1}$ and furthermore
\begin{itemize}
\item[i)]  $A_1\cup B_1=(A_0\cup B_0)\cup(m+A_0\cup B_0)$;
\item[ii)] $A_1\cap B_1\supseteq (A_0\cap B_0)\cup(m+A_0\cap B_0)$, the
    union being disjoint.
\end{itemize}
Moreover, if $m\notin(A_0-A_0)\cup(B_0-B_0)$, then also in i) the union is
disjoint, and in ii) the inclusion is in fact an equality.
% ; hence, in this
% case
%  $$ A_1\Del B_1 = (A_0\Del B_0) \cup (m + A_0\Del B_0). $$

In particular, if $A_0\cup B_0=[0,m-1]$, then $A_1\cup B_1=[0,2m-1]$, and if
$A_0$ and $B_0$ indeed \emph{partition} the interval $[0,m-1]$, then $A_1$
and $B_1$ partition the interval $[0,2m-1]$.
\end{lemma}

\begin{proof}
Since the assumption $m\notin A_0-B_0$ ensures that $A_0$ is disjoint
from $m+B_0$, for any integer $n$ we have
  $$ R_{A_1}(n) = R_{A_0}(n) + R_{B_0}(n-2m)
                 + | \{ (a_0,b_0)\in A_0\times B_0 \colon a_0+b_0=n-m\} |. $$
Similarly,
  $$ R_{B_1}(n) = R_{B_0}(n) + R_{A_0}(n-2m)
                + | \{ (a_0,b_0)\in A_0\times B_0 \colon a_0+b_0=n-m\} |, $$
and in view of $R_{A_0}=R_{B_0}$, this gives $R_{A_1}=R_{B_1}$. The
remaining assertions are straightforward to verify.
\end{proof}

Given subsets $A_0,B_0\seq\N_0$ and a sequence $(m_i)_{i\in\N_0}$ with
$m_i\in\N_0$ for each $i\in\N_0$, define subsequently
\begin{equation}\label{e:AB1}
  A_i := A_{i-1}\cup (m_{i-1}+B_{i-1})\ \text{and}
                  \ B_i:=B_{i-1}\cup(m_{i-1}+A_{i-1}),\quad i=1,2,\dotsc
\end{equation}
and let
\begin{equation}\label{e:AB2}
  A:=\cup_{i\in\N_0}A_i,\ B:=\cup_{i\in\N_0}B_i.
\end{equation}
As an immediate corollary of Lemma~\refl{main}, if $R_{A_0}=R_{B_0}$ and
$m_i\notin(A_i-B_i)\cup(B_i-A_i)$ for each $i\in\N_0$, then $R_A=R_B$.

The special case $A_0=\{0\}$, $B_0=\{1\}$, $m_i=2^{i+1}$ yields the
partition of Dombi (which, we remark, was originally expressed in
completely different terms). Below we analyze yet another special case
obtained by fixing arbitrarily an integer
 $l\ge 1$ and choosing $A_0:=\{0\}$, $B_0:=\{1\}$, and
\begin{equation}\label{e:m}
  m_i := \begin{cases}
              2^{i+1},           &\ 0\le i\le 2l-2,  \\
              2^{2l}-1,          &\ i=2l-1, \\
              2^{i+1}-2^{i-2l},  &\ i\ge 2l.
            \end{cases}
\end{equation}
We notice that $R_{A_0}=R_{B_0}$ in a trivial way (both functions are
identically equal to $0$), and that $A_0$ and $B_0$ partition the interval
$[0,m_0-1]$. Applying Lemma~\refl{main} inductively $2l-2$ times, we conclude
that in fact for each $i\le 2l-2$, the sets $A_i$ and $B_i$ partition the
interval $[0,2m_{i-1}-1]=[0,m_i-1]$, and consequently
$m_i\notin(A_i-B_i)\cup(B_i-A_i)$ and $m_i\notin(A_i-A_i)\cup(B_i-B_i)$. In
particular, $A_{2l-2}$ and $B_{2l-2}$ partition $[0,m_{2l-2}-1]$, and
therefore $A_{2l-1}$ and $B_{2l-1}$ partition $[0,2m_{2l-2}-1]=[0,m_{2l-1}]$.
In addition, it is easily seen that $A_{2l-1}$ contains both $0$ and
$m_{2l-1}$, whence $m_{2l-1}\in A_{2l-1}-A_{2l-1}$, but $m_{2l-1}\notin
B_{2l-1}-B_{2l-1}$ and
$m_{2l-1}\notin(A_{2l-1}-B_{2l-1})\cup(B_{2l-1}-A_{2l-1})$. From
Lemma~\refl{main} i) it follows now that
 $A_{2l}\cup B_{2l}=[0,2m_{2l-1}]=[0,m_{2l}-1]$, while
  $$ A_{2l}\cap B_{2l}=\Big( A_{2l-1}\cap(m_{2l-1}+A_{2l-1}) \Big)
      \cup \Big( B_{2l-1}\cap(m_{2l-1}+B_{2l-1}) \Big) = \{m_{2l-1}\}. $$
Applying again Lemma~\refl{main} we then conclude that for each
 $i\ge 2l$,
  $$ A_i \cup B_i = [0,m_i-1] $$
(implying $m_i\notin (A_i-B_i)\cup(B_i-A_i)$) and
  $$ A_i \cap B_i
           = m_{2l-1} + \{0,m_{2l},2m_{2l}\longc(2^{i-2l}-1)m_{2l} \}. $$

As a result, with $A$ and $B$ defined by~\refe{AB2}, we have
 $A\cup B=\N_0$ while the intersection of $A$ and $B$ is the infinite
arithmetic progression $m_{2l-1}+m_{2l}\N_0$. Moreover, the condition
$m_i\notin(A_i-B_i)\cup(B_i-A_i)$, which we have verified above to hold
for each $i\ge 0$, results in $R_A=R_B$.

We thus have proved
\begin{theorem}\label{t:main}
Let $l$ be a positive integer, and suppose that the sets $A,B\seq\N_0$
are obtained as in \refe{AB1}--\refe{AB2} starting from $A_0=\{0\}$ and
$B_0=\{1\}$, with $(m_i)$ defined by~\refe{m}. Then $R_A=R_B$, while
 $A\cup B=\N_0$ and $A\cap B=(2^{2l}-1)+(2^{2l+1}-1)\N_0$.
\end{theorem}

We notice that for any fixed integers $r\ge 2^{2l}-1$ and
 $m\ge 2^{2l+1}-1$, having \refe{m} appropriately modified (namely, setting
$m_i=2^{i-2l}m$ for $i\ge 2l$) and translating $A$ and $B$, one can
replace the progression $(2^{2l}-1)+(2^{2l+1}-1)\N_0$ in the statement of
Theorem~\reft{main} with the progression $r+m\N_0$; however, the relation
$A\cup B=\N_0$ will \emph{not} hold true any longer unless $r=2^{2l}-1$
and $m=2^{2l+1}-1$. This suggests the following question.
\begin{problem}
Given that $R_A=R_B$, $A\cup B=\N_0$, and $A\cap B=r+m\N_0$ with integer
$r\ge 0$ and $m\ge 2$, must there exist an integer $l\ge 1$ such that
$r=2^{2l}-1$, $m=2^{2l+1}-1$, and $A,B$ are as in Theorem~\reft{main}?
\end{problem}
The finite version of this question is as follows.
\begin{problem}
Given that $R_A=R_B$, $A\cup B=[0,m-1]$, and $A\cap B=\{r\}$ with integer
$r\ge 0$ and $m\ge 2$, must there exist an integer $l\ge 1$ such that
$r=2^{2l}-1$, $m=2^{2l+1}-1$, $A=A_{2l}$, and $B=B_{2l}$, with $A_{2l}$ nd
$B_{2l}$ as in the proof of Theorem~\reft{main}?
\end{problem}
We conclude our note with yet another natural problem.
\begin{problem}
Do there exist sets $A,B\seq\N_0$ with the infinite symmetric difference
and with $R_A=R_B$ which \emph{cannot} be obtained by a repeated
application of Lemma~\refl{main}?
\end{problem}

\subsection*{Acknowledgments}
Early stages of our work depended heavily on extensive computer programming
that was kindly carried out for us by Talmon Silver; we are indebted to him
for this contribution.

The first author is supported by the National Natural Science Foundation of
China, Grant No. 11371195, and the Priority Academic Program Development of
Jiangsu Higher Education Institutions (PAPD).

\bigskip

\end{document}